\documentclass{amsart} 

\usepackage[english]{babel}
\usepackage[latin1]{inputenc} 
\usepackage{amsmath}
\usepackage{amsthm}
\usepackage{amssymb}
\usepackage{tikz}

\usepackage{imakeidx}
\usepackage{appendix}
\usepackage{tikz-cd}
\usepackage{verbatim}
\usepackage{enumitem}
\usepackage{multicol}

\usepackage{hyperref}
\usepackage{cleveref}

\usepackage{mathtools}

\numberwithin{equation}{section} 
\newtheorem{thm}[equation]{Theorem} 
\newtheorem{question}[equation]{Question}
\newtheorem{prop}[equation]{Proposition}
\newtheorem{lemma}[equation]{Lemma}

\theoremstyle{definition}

\theoremstyle{remark}
\newtheorem{rmk}[equation]{Remark}

\newcommand{\F}{\mathbb F}
\newcommand{\Z}{\mathbb Z}
\renewcommand{\L}{\mathbb L}
\newcommand{\G}{\mathbb G}
\renewcommand{\P}{\mathbb P}

\renewcommand{\c}{\subseteq}

\newcommand{\A}{\mathbb A}

\newcommand{\mc}[1]{\mathcal{#1}}
\newcommand{\pre}[1]{\prescript{\alpha}{}{#1}}

\newcommand{\cl}{\overline}
\newcommand{\set}[1]{\{#1\}}

\renewcommand{\phi}{\varphi}
\newcommand{\on}[1]{\operatorname{#1}}

\DeclareMathOperator{\Spec}{Spec}

\DeclareMathOperator{\Gm}{\mathbb{G}_{\operatorname{m}}}

\newcommand{\ang}[1]{\left \langle{#1}\right \rangle}

\title{On the motivic class of an algebraic group}
\author{Federico Scavia}

\begin{document}
	\begin{abstract}
		Let $F$ be a field of characteristic zero admitting a biquadratic field extension. We give an example of a torus $G$ over $F$ whose classifying stack $BG$ is stably rational and such that $\set{BG}\neq \set{G}^{-1}$ in the Grothendieck ring of algebraic stacks over $F$. We also give an example of a finite \'etale group scheme $A$ over $F$ such that $BA$ is stably rational and $\set{BA}\neq 1$.
	\end{abstract}	
	\maketitle
	
	\section{Introduction}	
	Let $F$ be a field. The Grothendieck ring of algebraic stacks $K_0(\on{Stacks}_F)$ was introduced by Ekedahl in \cite{ekedahl2009grothendieck}, following up on earlier works \cite{behrend2007motivic}, \cite{joyce2007motivic}, \cite{toen2005grothendieck}. It is a variant of the Grothendieck ring of varieties $K_0(\on{Var}_F)$. By definition, $K_0(\on{Stacks}_F)$ is generated as an abelian group by the equivalence classes $\set{X}$ of all algebraic stacks $X$ of finite type over $F$ with affine stabilizers. These classes are subject to the scissor relations $\set{X}=\set{Y}+\set{{X}\setminus{Y}}$ for every closed substack $Y\c X$, and the relations $\set{E}=\set{\A^n\times X}$ for every vector bundle $E$ of rank $n$ over $X$. The product is defined by $\set{X}\cdot\set{Y}:=\set{X\times Y}$, and extended by linearity. 
	
	Given a group scheme $G$ over $F$, we may consider the class $\set{BG}$ of its classifying stack in $K_0(\on{Stacks}_F)$. The problem of computing $\set{BG}$ appears to be related to the problem of the stable rationality of $BG$, although no direct implications are known. Recall that $BG$ is stably rational if for one (equivalently, every) generically free representation $V$ of $G$, the rational quotient $V/G$ is stably rational. An equivalent terminology is that the Noether problem for stable rationality has a positive solution for $G$; see \cite[\S 3]{florence2018genus0}. The case of a finite (constant) group $G$ was considered in \cite{ekedahl2009geometric}: it frequently happens that $\set{BG}=1$ (notably for the symmetric groups, see \cite[Theorem 4.3]{ekedahl2009geometric}), although there are examples of finite groups $G$ for which $\set{BG}\neq 1$; see \cite[Corollary 5.2, Corollary 5.8]{ekedahl2009geometric}. Further work on the triviality of $\set{BG}$ for finite groups $G$ has been done in \cite{martino2016ekedahl} and \cite{martino2017introduction}. So far, all the known examples of finite group schemes $G$ for which $\set{BG}\neq 1$ are such that $BG$ is not stably rational. This suggests the following question.
	\begin{question}(cf. \cite[\S 6]{ekedahl2009geometric})\label{noetherfinite}
		Is it true that, for a finite group scheme $G$, the following two conditions are equivalent?		
		\begin{itemize}
			\item $BG$ is stably rational;
			\item $\set{BG}=1$ in $K_0(\on{Stacks}_F)$.
		\end{itemize}	
	\end{question}
	We will answer \Cref{noetherfinite} in the negative in \Cref{1.6}. 
	
	Now let $G$ be a connected linear algebraic group. Recall that $G$ is special if every $G$-torsor is Zariski-locally trivial. For example, $\on{GL}_n,\on{SL}_n$ and $\on{Sp}_n$ are special; see \cite{chevalley1958espaces}. It was shown by Ekedahl that if $P\to S$ is a torsor under the special group $G$, then $\set{P}=\set{G}\set{S}$. This is immediate if $S$ is a scheme, but less obvious when $S$ is a stack; see \cite[Corollary 2.4]{bergh2015motivic}. Applying this to the universal $G$-torsor $\Spec F\to BG$, one obtains $\set{BG}\set{G}=1$. 
	
	The equality $\set{BG}=\set{G}^{-1}$ appears to be the analogue for connected groups of the relation $\set{BG}=1$ for finite group schemes. In \cite{bergh2015motivic}, these equalities are referred to as \emph{expected class formulas}, and there is a sense in which they are  ``almost" true. In \cite[\S 2]{ekedahl2009grothendieck} Ekedahl defines a generalized Euler characteristic \[\chi_{\on{c}}:K_0(\on{Stacks}_F)\to K_0(\on{Coh}_F)\] taking values in a Grothendieck ring $K_0(\on{Coh}_F)$ of Galois representations over $F$. If $G$ is a finite group scheme, the equality $\chi_{\on{c}}(\set{BG})=1$ always holds \cite[Proposition 3.1]{ekedahl2009geometric}. On the other hand, if $G$ is connected, then $\chi_{\on{c}}(\set{BG}\set{G})=1$; see \cite[\S 2.2]{bergh2015motivic}. Since $\set{BG}\neq 1$ for some finite groups $G$, the following question naturally arises.
	
	\begin{question}\label{expected}
		Let $F$ be a field. Is it true that  \begin{equation}\label{exp}\set{BG}=\set{G}^{-1}\end{equation} in $K_0(\on{Stacks}_F)$ for every connected group $G$?
	\end{question}
	In \Cref{1.5}, we show that the answer to \Cref{expected} is also negative. Computations for non-special $G$ have been carried out for $\on{PGL}_2$ and $\on{PGL}_3$ in \cite{bergh2015motivic}, for $\on{SO}_n$ and $n$ odd in \cite{dhillon2016motive}, for $\on{SO}_n$ and $n$ even and $\on{O}_n$ for any $n$ in \cite{talpo2017motivic}, and for $\on{Spin}_7, \on{Spin}_8$ and $\on{G}_2$ in \cite{pirisi2017motivic}. In each of these cases, (\ref{exp}) was found to be true. The expectation was that, for a connected linear algebraic group $G$ over a field $F$ of characteristic $0$, \Cref{noether} below should have an affirmative answer. If $F$ is an algebraically closed field, then there are no examples of connected $G$ where $BG$ is known not to be stably rational. If $F$ is not assumed to be algebraically closed, then such examples exist. The following variant of \Cref{noetherfinite} seems natural in this context.
	\begin{question}(cf. \cite[\S 1]{talpo2017motivic} and \cite[Remark 4.1]{pirisi2017motivic})\label{noether}
		Is it true that, for a connected linear algebraic group $G$, the following two conditions are equivalent?		
		\begin{itemize}
			\item $BG$ is stably rational;
			\item $\set{BG}=\set{G}^{-1}$ in $K_0(\on{Stacks}_F)$.
		\end{itemize}	
	\end{question}
	
	Our first result gives a negative answer to \Cref{expected} and \Cref{noether}.
	
	\begin{thm}\label{1.5}
		Let $F$ be a field of characteristic zero which admits a biquadratic field extension $K$, let $E_1$ and $E_2$ be two distinct quadratic subextensions of $K/F$, and set $G:=R_{E_1\times E_2/F}^{(1)}(\Gm)$. Then
		\begin{enumerate}[label=(\alph*)]
			\item $BG$ is stably rational, and
			\item $\set{BG}\neq\set{G}^{-1}$ in $K_0(\on{Stacks}_F)$.
		\end{enumerate}
	\end{thm}
	
	The torus $G$ is an example of a norm-one torus; see \Cref{prelim} for the definition. It follows from \Cref{1.5} that counterexamples $H$ to (\ref{exp}) exist in any dimension $\dim H\geq 3$: consider for example $H:=G\times \G_{\on{m}}^r$ for $r\geq 0$.
	
	The key ingredient in the proof of \Cref{1.5} is the \emph{refined Euler characteristic} of Ekedahl, introduced in \cite[\S 6, 3]{ekedahl2009grothendieck}; see \Cref{sec4}.
	
	Our second result gives a negative answer to \Cref{noetherfinite}.
	\begin{thm}\label{1.6}
		Let $F$ be a field of characteristic zero which admits a biquadratic field extension $K$, and let $E_1$ and $E_2$ be two distinct quadratic subextensions of $K/F$. Define $G:=R_{E_1\times E_2/F}^{(1)}(\Gm)$, and let $A:=G[2]$ be the $2$-torsion subgroup of $G$. Then
		\begin{enumerate}[label=(\alph*)]
			\item $BA$ is stably rational, and
			\item $\set{BA}\neq 1$ in $K_0(\on{Stacks}_F)$.
		\end{enumerate}
	\end{thm}
	
	Questions \ref{noetherfinite}, \ref{expected} and \ref{noether} remain open in the case, where the base field $F$ is assumed to be algebraically closed. Our arguments do not shed any new light in this setting.

	The remainder of this paper is structured as follows. In \Cref{prelim} we review well known computations of motivic classes for non-split tori. In \Cref{sec3} we obtain explicit formulas for the motivic classes of $G$ and $BG$, and in \Cref{sec4} we give the required background on the refined Euler characteristic.
	In \Cref{n2} we prove \Cref{1.5}, and in \Cref{sec5} we prove \Cref{1.6}.

	\section{Preliminaries}\label{prelim}
	Let $F$ be a field. We will write $\L$ for the class $\set{\A^1}$ in $K_0(\on{Var}_F)$ or $K_0(\on{Stacks}_F)$. If $E$ is an \'etale algebra over $F$, we will denote by $\set{E}$ the class $\set{\Spec E}$ in $K_0(\on{Var}_F)$ or $K_0(\on{Stacks}_F)$. If $X$ is a quasi-projective scheme over $E$, we will denote by $R_{E/F}(X)$ the \emph{Weil restriction} of $X$ to $F$. By definition, for every $F$-scheme $S$ one has $R_{E/F}(X)(S)=X(S_E)$. We refer the reader to \cite[\S 3.12]{voskresenskii2011algebraic} for an account of the main properties of the Weil restriction.
	
	Let $G$ be a linear algebraic group over $F$, and $\alpha\in H^1(F,G)$ be represented by a $G$-torsor $P\to \Spec F$. For every quasi-projective $F$-scheme $Z$, we denote by $\pre{Z}$ the \emph{twist} of $Z$ by $P$, that is, \[\pre{Z}:=(Y\times P)/G,\] where $G$ acts diagonally. We refer the reader to \cite[Section 2]{florence2008inventiones} for the definition and the basic properties of the twisting operation. 
	
	We will write $C_2$ for the cyclic group of two elements, and $S_n$ for the symmetric group on $n$ symbols.
	
	The following observations will be repeatedly used in the sequel.
	
	\begin{lemma}\label{norat}
		Let $X$ be a scheme over $F$, $E$ an \'etale algebra of degree $n$ over $F$, $\alpha\in H^1(F,S_n)$ the class corresponding to $E/F$.
		\begin{enumerate}[label=(\alph*)] 
			\item Let $S_n$ act on the disjoint union $\amalg_{i=1}^n X$ by permuting the $n$ copies of $X$. Then \[\pre{(\amalg_{i=1}^n X)}\cong X_E.\]
			\item Let $S_n$ act on $X^n$ by permuting the $n$ factors. Then \[\pre{(X^n)}\cong R_{E/F}(X).\]
		\end{enumerate}
	\end{lemma}
	
	\begin{proof}
		(a) Let $Y:=\amalg_{i=1}^n X$, and let $S_n$ act on $Y$ by permuting the copies of $X$. By definition, \[\pre{Y}=(Y\times \Spec E)/S_n\cong (Y\times_X X_E)/S_n,\] where $S_n$ acts diagonally. This shows that $\pre{Y}$ is the twist of $X_E$ by the trivial $S_n$-torsor $Y\to X$ in the category of $X$-schemes, which implies $\pre{Y}\cong X_E$.
		
		(b) See the bottom of page 5 in \cite{florence2017positivegenus}. 
	\end{proof}
	
	\begin{lemma}\label{bergh}
		Let \[1\to N\to G\to H\to 1 \] be an exact sequence of group schemes over $F$, and assume that $G$ is special. Then \[\set{BN}=\set{H}/\set{G}.\]
	\end{lemma}
	
	\begin{proof}
		See \cite[Proposition 2.9]{bergh2015motivic}.
	\end{proof}
	
	Let $F_s$ be a separable closure of $F$. Recall that a group scheme $T$ over $F$ is called a \emph{torus} if $T_{F_s}\cong \G_{\on{m},F_s}^n$ for some $n\geq 0$. The \emph{character lattice} of $T$ is the finitely generated $\Z$-free $\on{Gal}(F)$-module $\on{Hom}_{F_s}(T_{F_s},\G_{\on{m},F_s})$. The character lattice induces an anti-equivalence between the category of $F$-tori and the category of $\on{Gal}(F)$-lattices, i.e., $\Z$-free continuous $\on{Gal}(F)$-modules; see \cite[\S 2]{favi2008tori}. Similarly, for every separable finite extension $L/F$, we have an anti-equivalence between $\on{Gal}(L/F)$-lattices and $F$-tori $T$ split by $L$, i.e., such that $T_L\cong \G_{\on{m},L}^n$ for some $n\geq 0$. The \emph{dual torus} of $T$ is the torus $T'$ whose character lattice is dual to that of $T$.
	
	Let $E$ be an \'etale algebra over $F$. If $G$ is a group scheme over $E$, then $R_{E/F}(G)$ is a group scheme over $F$. The group $R_{E/F}(\Gm):=R_{E/F}(\G_{\on{m},E})$ is an $F$-torus. Tori of this kind are called \emph{quasi-split}. They are special groups, and they correspond to permutation $\on{Gal}(F)$-lattices, that is, lattices admitting a $\Z$-basis that is permuted by $\on{Gal}(F)$; see \cite[\S 3.12, Example 19]{voskresenskii2011algebraic}. 
	
	\begin{lemma}\label{stablyrat}
		Let $T$ be an algebraic torus over $F$, and let $T'$ be its dual. Assume that $T$ is stably rational. Then 
		\begin{enumerate}[label=(\alph*)]
			\item $BT'$ is stably rational;
			\item $\set{BT'}\set{T}=1$ in $K_0(\on{Stacks}_F)$.
		\end{enumerate}
	\end{lemma}
	
	\begin{proof}
		Since $T$ is stably rational, by \cite[\S 4.7, Theorem 2]{voskresenskii2011algebraic} there is a short exact sequence \begin{equation}\label{qsplit}1\to T_1\to T_2\to T\to 1\end{equation} where $T_1$ and $T_2$ are quasi-split. Since quasi-split tori are isomorphic to their dual, the sequence dual to (\ref{qsplit}), \begin{equation}\label{qsplit2}1\to T'\to T_2\to T_1\to 1,\end{equation} shows that $T'$ embeds in $T_2$. We may view $T_2$ as a maximal torus inside $\on{GL}_n$, where $n=\on{rank}T_2$. This gives a faithful representation of $T'$ with quotient birational to $T_1$. Since quasi-split tori are rational, it follows that $BT'$ is stably rational. 
		
		Quasi-split tori are special, so we may apply \Cref{bergh} to (\ref{qsplit}) and (\ref{qsplit2}). We obtain $\set{T}=\set{T_2}/\set{T_1}$ and $\set{BT'}=\set{T_1}/\set{T_2}$, so $\set{BT'}\set{T}=1$.
	\end{proof}
	
	Let $E/F$ be an \'etale algebra, and let $R_{E/F}(\Gm)$ be the associated quasi-split torus. The kernel of the norm homomorphism $R_{E/F}(\Gm)\to \Gm$ is called a \emph{norm-one torus}, and is denoted by $R^{(1)}_{E/F}(\Gm)$. Its dual torus is isomorphic to $R_{E/F}(\Gm)/\Gm$.
	
	\begin{lemma}\label{torideg2}
		Assume that $\on{char}F\neq 2$. Let $E:=F(\sqrt{m})$ be a separable quadratic field extension, and let $\alpha$ denote the class of $E/F$ in $H^1(F,C_2)$. Then:
		\begin{enumerate}[label=(\alph*)] 
			\item $R^{(1)}_{E/F}(\Gm)\cong R_{E/F}(\Gm)/\Gm$.
			\item Let $\on{Gal}(E/F)$ act on $\P^1$ via $z\mapsto z^{-1}$. Then $\pre{\P^1}\cong \P^1$.	
			\item $R_{E/F}(\Gm)/\Gm$ is rational and \[\set{R_{E/F}(\Gm)/\Gm}=\set{B(R_{E/F}(\Gm)/\Gm)}^{-1}=\L-\set{E}+1.\]
			\item $\set{R_{E/F}(\Gm)}=\set{BR_{E/F}(\Gm)}^{-1}=(\L-1)(\L-\set{E}+1)$.
			\item $\set{R_{E/F}(\P^1)}=\L^2+\set{E}\L+1$.
		\end{enumerate}	
	\end{lemma}
	
	\begin{proof}
		(a) Both tori correspond to the unique non-trivial $\on{Gal}(E/F)$-lattice of rank $1$. Here $\on{Gal}(E/F)\cong C_2$.
		
		(b) The $C_2$-action on $\P^1$ has a fixed point $z = 1$, hence $\pre{\P^1}$ has an $F$-point. By Ch\^{a}telet's Theorem \cite[Theorem 5.1.3]{gille_szamuely_2006}, a form of $\P^n$ which admits an $F$-point is trivial (the case $n=1$ is particularly simple, see \cite[Remark 1.3.5]{gille_szamuely_2006}). We conclude that $\pre{\P^1}\cong \P^1$.
		
		(c) Let $T:=R^{(1)}_{E/F}(\Gm)\cong R_{E/F}(\Gm)/\Gm$. The open embedding $\Gm\hookrightarrow \P^1$, as the complement of $Z:=\set{0,\infty}$, is equivariant under the $C_2$-action on $\Gm$ and $\P^1$ given by $z\mapsto z^{-1}$. Twisting by $\alpha$, we obtain by (b) an open embedding of $T$ in $\P^1$ as the complement of $\pre{Z}$. In particular, $T$ is rational. By \Cref{norat}(a), $\pre{Z}\cong\Spec E$, so \[\set{T}=\set{\P^1}-\set{\pre{Z}}=\L+1-\set{E}.\]
		Now (c) follows from \Cref{stablyrat}(b).
		
		(d) The first equality holds because $R_{E/F}(\Gm)$ is special. Consider the short exact sequence \[1\to \Gm\to R_{E/F}(\Gm)\to T\to 1.\] Since $R_{E/F}(\Gm)$ is special, \Cref{bergh} yields \[\set{R_{E/F}(\Gm)}=(\L-1)\set{BT}^{-1},\] thus (d) follows from (c). 
		
		(e) Write $\P^1=\A^1\cup\set{\infty}$, and consider the $C_2$-equivariant decomposition \[(\P^1)^2=(\A^1)^2\amalg (\A^1\times\set{\infty}\cup \set{\infty}\times\A^1)\amalg \set{(\infty,\infty)}.\] By Hilbert's Theorem 90 and \Cref{norat}(a), twisting by $\alpha$ gives \[R_{E/F}(\P^1)=\A^2\amalg \A^1_E\amalg\Spec F,\] thus $\set{R_{E/F}(\P^1)}=\L^2+\set{E}\L+1$.	\end{proof}

	\section{The classes of $G$ and $BG$}\label{sec3}
	Let $F$ be a field of characteristic not $2$, and assume that there exists a biquadratic extension \[K:=F(\sqrt{m_1},\sqrt{m_2})\] of $F$. Let \[E_1:=F(\sqrt{m_1}),\qquad E_2:=F(\sqrt{m_2}),\qquad E_{12}:=F(\sqrt{m_1m_2}), \qquad E:=E_1\times E_2,\] and let $\Gamma:=\on{Gal}(K/F)\cong C_2^2$ be the Galois group of $K/F$. We define the torus \[G:=R^{(1)}_{E/F}(\Gm)\] and let \[G':=R_{E/F}(\Gm)/\Gm\] be the dual torus of $G$. By definition, we have a short exact sequence \begin{equation}\label{normone}
	   1\to G\to R_{E/F}(\Gm)\xrightarrow{N} \Gm\to 1,
	\end{equation} where $N$ is the norm homomorphism. 
	
	The purpose of this section is the proof of \Cref{dual'}, which expresses $\set{BG}$ and $\set{G}$ as rational functions in $\L$, with coefficients classes of \'etale algebras.
	
	Let $\sigma_1$ and $\sigma_2$ be generators for $\Gamma$ such that $E_1=K^{\sigma_1}$ and $E_2=K^{\sigma_2}$. Consider the $\Gamma$-action on $\G_{\on{m}}^2$, where $\sigma_1(u,v)=(v^{-1},u^{-1})$ and $\sigma_2(u,v)=(v,u)$, and set \begin{equation}\label{rank2}T:=\pre{(\G_{\on{m}}^2)},\end{equation} where $\alpha\in H^1(F,\Gamma)$ corresponds to the extension $K/F$.
	
	\begin{lemma}\label{t}
		We have \[\set{T}=\L^2+(\set{E_{12}}-\set{K})\L+\set{K}-\set{E_1}-\set{E_2}+1.\]
	\end{lemma}
	
	\begin{proof}
		The embedding of $\Gm$ in $\P^1$ as the complement of $Z:=\set{0,\infty}$ gives an open embedding $\G_{\on{m}}^2\hookrightarrow (\P^1)^2$ such that the $\Gamma$-action on $\G_{\on{m}}^2$ extends to $(\P^1)^2$.
		By definition \[\pre{(\P^1)}^2=((\P^1)^2\times \Spec K)/\Gamma,\] where $\Gamma=\ang{\sigma_1,\sigma_2}$ acts diagonally. We first take the quotient by the subgroup $\ang{\sigma_1\sigma_2}$. Since $\sigma_1\sigma_2(u,v)=(u^{-1},v^{-1})$ and $E_{12}=K^{\sigma_1\sigma_2}$, by \Cref{torideg2}(b)
		\[\pre{(\P^1)^2}=((\P^1)^2\times \Spec E_{12})/C_2,\] where $C_2$ acts on $(\P^1)^2$ by switching the two factors. Here we are using the fact that every automorphism of $(\P^1)^2$ must respect the ruling (because it respects the intersection form), and so $\on{Aut}((\P^1)^2)= (\on{Aut}(\P^1))^2\rtimes C_2$, where $C_2$ switches the two factors. By \Cref{norat}(b) we deduce that $\pre{(\P^1)^2}\cong R_{E_{12}/F}(\P^1)$, so by \Cref{torideg2}(e) \begin{equation}\label{twistproj}\set{\pre{(\P^1)^2}}=\L^2+\set{E_{12}}\L+1.\end{equation}
		We may partition $(\P^1)^2\setminus \G_{\on{m}}^2$ in two strata \[Z_1:=Z\times Z,\qquad  Z_2:=(Z\times \Gm)\amalg (\Gm\times Z).\]
		The $\Gamma$-action on $Z_1$ has two orbits, and $\Gamma$ acts on $Z_2$ by transitively permuting the components as the Klein subgroup of $S_4$. By \Cref{norat}(a), $\pre{Z_1}=\Spec E_1\amalg \Spec E_2$ and $\pre{Z_2}=\Gm\times\Spec K$. By (\ref{twistproj})
		\begin{align*}\set{T}&=\set{\pre{(\P^1)}^2}-\set{\pre{Z_1}}-\set{\pre{Z_2}}\\ &=\L^2+\set{E_{12}}\L+1-\set{E_1}-\set{E_2}-\set{K}(\L-1) \\ &=\L^2+(\set{E_{12}}-\set{K})\L+\set{K}-\set{E_1}-\set{E_2}+1.\qedhere\end{align*}
	\end{proof}
	
	\begin{prop}\label{nicepres}
	There is a short exact sequence of tori
	\[1\to \Gm\to G\to T\to 1,\] where $T$ is the torus of (\ref{rank2}).
	\end{prop}
	
	\begin{proof}
		Let $P$, $M$ and $\Z$ be the character lattices of $R_{E/F}(\Gm)$, $G$ and $\Gm$, respectively. We may view $P$ as the $\Gamma$-lattice with a basis $e_1,e_2,e_3,e_4$, such that $\sigma_1$ acts by switching $e_1$ with $e_2$ and fixing $e_3$ and $e_4$, and $\sigma_2$ switches $e_3$ with $e_4$ and fixes $e_1$ and $e_2$. The sequence of $\Gamma$-lattices dual to (\ref{normone}) identifies $M$ with the cokernel of the $\Gamma$-homomorphism $\Z\to P$ given by $1\mapsto e_1+e_2+e_3+e_4$; denote by $\cl{e}_i\in M$ the projection of $e_i$. Following Kunyavski\u{\i} \cite[\S 3, Proposition 1(b)]{kunyavskii1987rank}, we consider an exact sequence of $\Gamma$-lattices \begin{equation}\label{lattice2}0\to N\to M\xrightarrow{\pi} \Z\to 0.\end{equation} The map $\pi$ is defined by $\pi(\sum a_i\cl{e}_i)=a_1+a_2-a_3-a_4$, and $N:=\on{Ker}\pi$. A basis for $N$ is given by $v_1:=\cl{e}_1+\cl{e}_3$ and $v_2:=\cl{e}_1+\cl{e}_4$. With respect to the basis $(v_1,v_2)$, the $\Gamma$-action on $N$ is given by $\sigma_1(a,b)=(-b,-a)$ and $\sigma_2(a,b)=(b,a)$. It is now clear that $N$ is the character lattice of the torus $T$ of (\ref{rank2}), hence the proof is complete.	
	\end{proof}
	
	\begin{prop}\label{dual'}
		\begin{enumerate}[label=(\alph*)]
			\item $BG$ is stably rational.
			\item $\set{BG}\set{G'}=1$ in $K_0(\on{Stacks}_F)$.
		\end{enumerate}
	\end{prop}
	
	\begin{proof}
		Consider the sequence \begin{equation}\label{ebbasta}1\to \Gm\to G'\to (R_{E_1/F}(\Gm)/\Gm)\times (R_{E_2/F}(\Gm)/\Gm)\to 1,\end{equation} which exhibits $G'$ as a $\Gm$-torsor over a rational variety, by \Cref{torideg2}(c). We deduce that $G'$ is rational, and now (a) and (b) follow from \Cref{stablyrat}. 	
	\end{proof}
	
	\begin{prop}\label{ginv}
		We have
		\begin{equation}\label{ginv2}\set{G}=(\L-1)(\L^2+(\set{E_{12}}-\set{K})\L+\set{K}-\set{E_1}-\set{E_2}+1)\end{equation} and
		\begin{equation}\label{ginv1}\set{BG}^{-1}=(\L-1)(\L-\set{E_1}+1)(\L-\set{E_2}+1)\end{equation}
		in $K_0(\on{Stacks}_F)$.
	\end{prop} 
	
	\begin{proof}
		By \Cref{nicepres}, $G$ is a $\Gm$-torsor over $T$. Since $\Gm$ is special, $\set{G}=(\L-1)\set{T}$. The class of $T$ was determined in \Cref{t}.
		
		By \Cref{dual'}(b), $\set{BG}^{-1}=\set{G'}$. Since $\Gm$ is special, by (\ref{ebbasta}), $\set{G'}=(\L-1)\set{R_{E_1/F}^{(1)}(\Gm)}\set{R_{E_2/F}^{(1)}(\Gm)}$. Now (\ref{ginv1}) follows from \Cref{torideg2}(c). 
	\end{proof}
	
	\section{The refined Euler characteristic}\label{sec4}
	Let $F$ be a field of characteristic zero. Using the computations of the previous section, we will reduce \Cref{1.5}(b) to the assertion that a certain polynomial in $\L$ with coefficients motivic classes of \'etale algebras  is a non-zero element of $K_0(\on{Var}_F)$. To prove the assertion, we will use a simplified version of the refined Euler characteristic, introduced by Ekedahl in \cite{ekedahl2009grothendieck}. 
    
    Fix a prime number $p$, and let $\mc{G}$ be a profinite group. The \emph{representation ring} $a_p(\mc{G})$ of $\mc{G}$ is the Grothendieck ring of continuous $\mc{G}$-representations $[M]$ of finite dimension over $\F_p$, subject to the relations $[M\oplus N]=[M]+[N]$. Note that no relations for non-split short exact sequences are imposed. The product structure on $a_p(\mc{G})$ is given by tensor product of representations. The next observation is well known when $\mc{G}$ is assumed to be finite; see \cite[\S 5.1]{benson1998representations}.
    
    \begin{lemma}\label{sameclass}
    As an abelian group, $a_p(\mc{G})$ is freely generated by the set of isomorphism classes of indecomposable representations.
	\end{lemma}
	
	\begin{proof}
	It is clear that $a_p(\mc{G})$ is generated by isomorphism classes of indecomposable representations. Assume that $\sum a_i[M_i]-\sum b_j[N_j]=0$ in $a_p(\mc{G})$, for some positive integers $a_i,b_j$ and some pairwise non-isomorphic indecomposable $\mc{G}$-representations $M_i$ and $N_j$.
	
	As a group, $a_p(\mc{G})$ is the quotient group $F/I$, where $F$ is the free abelian group with one generator $\ang{P}$ for every isomorphism class of $\mc{G}$-representations $P$, and $I$ is the subgroup generated by all elements of the form $\ang{P\oplus Q}-\ang{P}-\ang{Q}$. It follows that we may find a $\mc{G}$-representation $X$ such that
	\[(\oplus_i M_i^{\oplus a_i})\oplus X\cong (\oplus_j N_j^{\oplus b_j})\oplus X.\]
	Let $\mc{G}_0$ be a finite quotient of $\mc{G}$ such that $\mc{G}$ acts on $M_i$, $N_j$ and $X$ through $\mc{G}_0$. Then $M\oplus X\cong N\oplus X$ as $\mc{G}_0$-representations. By the Krull-Schmidt Theorem applied to the group algebra $\F_p[\mc{G}_0]$, this implies $M\cong N$ as $\mc{G}_0$-modules, hence as $\mc{G}$-modules. This is impossible, because the indecomposable representations $M_i$ and $N_j$ are pairwise non-isomorphic.
	\end{proof}
    
	\begin{prop}
	Let $F$ be a field of characteristic zero, let $\on{Gal}(F)$ be the absolute Galois group of $F$, and let $R_p:=a_p(\on{Gal}(F))$. There is a ring homomorphism \[\label{mu}\mu:K_0(\on{Var}_F)\to R_p[t]\] such that for every smooth complete variety $X$ we have $\mu(X)=\sum_i[H^{i}(\cl{X}_{\text{\'et}},\F_p)]t^i$.
	\end{prop}

	\begin{proof}
	See the proof of \cite[Proposition 3.2(i)]{ekedahl2009grothendieck}. To show that $\mu$ is well-defined, one needs to assume that $\on{char}F=0$ in order to invoke Bittner's presentation of $K_0(\on{Var}_F)$; see \cite[Theorem 3.1]{bittner2004universal}.
	\end{proof}

	\section{Proof of Theorem \ref{1.5}}\label{n2}
	
	\Cref{1.5}(a) was proved in \Cref{dual'}(b), so we will focus on \Cref{1.5}(b). We maintain the notation given at the beginning of \Cref{sec3}.

	\begin{proof}[Proof of \Cref{1.5}(b)]
		Assume by contradiction that $G=R^{(1)}_{E/F}(\Gm)$ satisfies (\ref{exp}). Then by \Cref{ginv} we have \begin{align*}(\L-1)(\L-\set{E_1}+1)(\L-\set{E_2}+1)= \\ =(\L-1)(\L^2+(\set{E_{12}}-\set{K})\L+\set{K}-\set{E_1}-\set{E_2}+1)\end{align*} in $K_0(\on{Stacks}_F)$. Since $\L-1$ is invertible in $K_0(\on{Stacks}_F)$, we may divide by $\L-1$ on both sides. Subtracting $\L^2$ on the left and on the right, we arrive to
		\[(2-\set{E_1}-\set{E_2})\L+(1-\set{E_1})(1-\set{E_2})=(\set{E_{12}}-\set{K})\L+\set{K}-\set{E_1}-\set{E_2}+1,\] that is \[(\set{K}-\set{E_1}-\set{E_2}-\set{E_{12}}+2)\L=0\] in $K_0(\on{Stacks}_F)$. 
		
		Recall that $K_0(\on{Stacks}_F)$ is the localization of $K_0(\on{Var}_F)$ at $\L$ and the cyclotomic polynomials in $\L$; see \cite[Theorem 1.2]{ekedahl2009grothendieck}. It follows that \begin{equation}\label{finaleq}
		(\set{K}-\set{E_1}-\set{E_2}-\set{E_{12}}+2)f(\L)=0    
		\end{equation} in $K_0(\on{Var}_F)$, where $f(x)\in \Z[x]$ is a monic polynomial of some degree $n$. 
		
		In order to obtain a contradiction, we now want to apply the homomorphism $\mu$ of (\ref{mu}), with respect to the prime $p=2$. If $L/F$ is an \'etale algebra of degree $n$, $\mu(\set{L})$ consists of the permutation representation of $\on{Gal}(F)$ associated to $L$, concentrated in degree $0$. Since we have chosen $p=2$, $\mu(\set{\P^1})$ consists of one copy of the trivial representation in degree $0$ and $2$ (in the case $p>2$ one would need a Tate twist in degree $2$). Since $\L=\set{\P^1}-1$, we deduce that $\mu(\L)=t^2$, and hence $\mu(f(\L))=f(t^2)$. 
		
	If $X$ is a finite $\on{Gal}(F)$-set, we denote by $\F_2[X]$ the permutation representation over $\F_2$ associated to $X$. Recall from \Cref{sec3} that we denote $\on{Gal}(K/F)$ by $\Gamma=\ang{\sigma_1,\sigma_2}$. Applying $\mu$ to (\ref{finaleq}) and looking at degree $2n$, we obtain
	\[[\F_2[\Gamma]]-[\F_2[\Gamma/\ang{\sigma_1}]]-[\F_2[\Gamma/\ang{\sigma_2}]]-[\F_2[\Gamma/\ang{\sigma_{12}}]]+2[\F_2]=0\] in $R_2$. This is a non-trivial relation of linear dependence in $R_2$ among classes of indecomposable representations. This is in contradiction with \Cref{sameclass}, hence $\set{BG}\neq \set{G}^{-1}$, as desired.
	\end{proof}
	
	\begin{rmk}
		By \cite[\S 4.9, Example 7]{voskresenskii2011algebraic} every torus of rank $2$ is rational, so by \Cref{nicepres} the torus $G$ is rational. By \Cref{stablyrat}, $BG'$ is stably rational and $\set{BG'}=\set{G}^{-1}$. By \Cref{dual'}(b) we have $\set{BG}=\set{G'}^{-1}$, so $\set{BG'}\set{G'}=\set{BG}^{-1}\set{G}^{-1}$. Since $\set{BG}\set{G}\neq 1$, the conclusions of \Cref{1.5}(a) and (b) hold for $G'$ as well. 
	\end{rmk}

	\section{Proof of Theorem \ref{1.6}}\label{sec5}
We maintain the notation of \Cref{sec3}.

	\begin{proof}[Proof of Theorem \ref{1.6}]
	Let $\Gamma:=\on{Gal}(K/F)$, let $M$ be the character lattice of $G$, so that $M/2M$ is the character module of $A$, and let $P$ be the character lattice of $R_{E/F}(\Gm)$. As in the proof of \Cref{nicepres}, we view $P$ as the lattice freely generated by $e_1,e_2,e_3,e_4$, such that $\sigma_1$ acts by switching $e_1$ with $e_2$, and $\sigma_2$  by switching $e_3$ with $e_4$. Using (\ref{normone}), we may construct a commutative diagram of $\Gamma$-modules
		\begin{equation}\label{diagram16}
		\begin{tikzcd}
		0 \arrow[r] & \Z \arrow[r] \arrow[d,"\iota"] & P \arrow[r] \arrow[equal]{d} & M \arrow[r] \arrow[d] & 0\\
		0 \arrow[r] & N \arrow[r] & P \arrow[r,"\phi"] & M/2M\arrow[r] & 0.
		\end{tikzcd}
		\end{equation}
	with exact rows. Here $\Z$ denotes the trivial one-dimensional $\Gamma$-lattice, $\iota(1):=e_1+e_2+e_3+e_4$, and $N$ is the kernel of $\phi$, that is,
	\[N=\set{\sum_{i=1}^4a_ie_i: a_1\equiv a_2\equiv a_3\equiv a_4\pmod 2}.\]
	Applying the snake lemma to (\ref{diagram16}), we obtain a short exact sequence
	\[0\to\Z \xrightarrow{\iota} N\to M\to 0.\]
	Define $\pi:N\to \Z$ by sending $\sum a_ie_i$ to $(a_1+a_2)/2$. Then $\pi$ is a $\Gamma$-homomorphism and $\iota$ is a section of $\pi$. Therefore, we have an isomorphism $N\cong \Z\oplus M$.
	
	Let $S$ be an $F$-torus with character lattice $N$. Since $N\cong \Z\oplus M$, we have $S\cong\G_{\on{m}}\times G$. The bottom row of (\ref{diagram16}) corresponds to the short exact sequence of group schemes \[1\to A\to R_{E/F}(\Gm)\to \G_{\on{m}}\times G\to 1.\]
	By \Cref{bergh}, we have $\set{BA}=\set{\Gm}\set{G}/\set{R_{E/F}(\Gm)}$. Applying \Cref{bergh} to (\ref{normone}), we see that $\set{BG}=\set{\Gm}/\set{R_{E/F}(\Gm)}$. Therefore, $\set{BA}=\set{BG}\set{G}$. By \Cref{1.6} we have $\set{BG}\neq \set{G}^{-1}$, hence $\set{BA}\neq 1$, as desired.
\end{proof}
	
	\section*{Acknowledgments}
	I would like to thank my advisor Zinovy Reichstein for his guidance and for greatly improving the exposition, Mattia Talpo and Angelo Vistoli for helpful comments, and Boris Kunyavski\u{\i} for sending me a copy of his paper \cite{kunyavskii1987rank}. I am very grateful to the anonymous referee for finding a mistake in a previous version of the proof of \Cref{1.5}, and for suggesting a fix.


\begin{thebibliography}{10}
		
		\bibitem{behrend2007motivic}
		Kai Behrend and Ajneet Dhillon.
		\newblock On the motivic class of the stack of bundles.
		\newblock {\em Advances in Mathematics}, 212(2):617--644, 2007.
		
		\bibitem{benson1998representations}
		D.~J. Benson.
		\newblock {\em Representations and cohomology. {II}}, volume~31 of {\em
			Cambridge Studies in Advanced Mathematics}.
		\newblock Cambridge University Press, Cambridge, second edition, 1998.
		\newblock Cohomology of groups and modules.
		
		\bibitem{bergh2015motivic}
		Daniel Bergh.
		\newblock Motivic classes of some classifying stacks.
		\newblock {\em Journal of the London Mathematical Society}, 93(1):219--243,
		2015.
		
		\bibitem{bittner2004universal}
		Franziska Bittner.
		\newblock The universal Euler characteristic for varieties of characteristic
		zero.
		\newblock {\em Compositio Mathematica}, 140(4):1011--1032, 2004.
		
		\bibitem{chevalley1958espaces}
		S{\'e}minaire~Claude Chevalley and JP~Serre.
		\newblock Espaces fibr{\'e}s alg{\'e}briques.
		\newblock {\em S{\'e}minaire Claude Chevalley}, 3:1--37, 1958.
		
		\bibitem{dhillon2016motive}
		Ajneet Dhillon, Matthew~B Young, et~al.
		\newblock The motive of the classifying stack of the orthogonal group.
		\newblock {\em The Michigan Mathematical Journal}, 65(1):189--197, 2016.
		
		\bibitem{ekedahl2009geometric}
		Torsten Ekedahl.
		\newblock A geometric invariant of a finite group.
		\newblock {\em arXiv preprint arXiv:0903.3148}, 2009.
		
		\bibitem{ekedahl2009grothendieck}
		Torsten Ekedahl.
		\newblock The {G}rothendieck group of algebraic stacks.
		\newblock {\em arXiv preprint arXiv:0903.3143}, 2009.
		
		\bibitem{favi2008tori}
		Giordano Favi and Mathieu Florence.
		\newblock Tori and essential dimension.
		\newblock {\em J. Algebra}, 319(9):3885--3900, 2008.
		
		\bibitem{florence2008inventiones}
		Mathieu Florence.
		\newblock On the essential dimension of cyclic {$p$}-groups.
		\newblock {\em Invent. Math.}, 171(1):175--189, 2008.
		
		\bibitem{florence2017positivegenus}
		Mathieu Florence and Zinovy Reichstein.
		\newblock On the rationality problem for forms of moduli spaces of stable
		marked curves of positive genus.
		\newblock {\em arXiv preprint arXiv:1709.05696}, 2017.
		
		\bibitem{florence2018genus0}
		Mathieu Florence and Zinovy Reichstein.
		\newblock The rationality problem for forms of ${M}_{0,n}$.
		\newblock {\em Bulletin of the London Mathematical Society}, 50(1):148--158,
		2018.
		
		\bibitem{gille_szamuely_2006}
		Philippe Gille and Tam{\'a}s Szamuely.
		\newblock {\em Central Simple Algebras and Galois Cohomology}.
		\newblock Cambridge Studies in Advanced Mathematics. Cambridge University
		Press, 2006.
		
		\bibitem{joyce2007motivic}
		Dominic Joyce.
		\newblock Motivic invariants of {A}rtin stacks and stack functions.
		\newblock {\em The Quarterly Journal of Mathematics}, 58(3):345--392, 2007.
		
		\bibitem{kunyavskii1987rank}
		B.~\`E. Kunyavski\u\i.
		\newblock Three-dimensional algebraic tori.
		\newblock In {\em Investigations in number theory ({R}ussian)}, pages 90--111.
		Saratov. Gos. Univ., Saratov, 1987.
		\newblock Translated in Selecta Math. Soviet. {{\bf{9}}} (1990), no. 1, 1--21.
		
		\bibitem{martino2016ekedahl}
		Ivan Martino.
		\newblock The {E}kedahl invariants for finite groups.
		\newblock {\em Journal of Pure and Applied Algebra}, 220(4):1294--1309, 2016.
		
		\bibitem{martino2017introduction}
		Ivan Martino.
		\newblock Introduction to the {E}kedahl {I}nvariants.
		\newblock {\em Mathematica scandinavica}, 120(2):211--224, 2017.
		
		\bibitem{pirisi2017motivic}
		Roberto Pirisi and Mattia Talpo.
		\newblock On the motivic class of the classifying stack of ${G}_2$ and the spin
		groups.
		\newblock {\em To appear in International Mathematics Research Notices. arXiv
			preprint arXiv:0903.3143}.
		
		\bibitem{talpo2017motivic}
		Mattia Talpo and Angelo Vistoli.
		\newblock The motivic class of the classifying stack of the special orthogonal
		group.
		\newblock {\em Bulletin of the London Mathematical Society}, 49(5):818--823,
		2017.
		
		\bibitem{toen2005grothendieck}
		Bertrand To{\"e}n.
		\newblock Grothendieck rings of {A}rtin n-stacks.
		\newblock {\em arXiv preprint math/0509098}, 2005.
		
		\bibitem{voskresenskii2011algebraic}
		V.~E. Voskresenski\u{\i}.
		\newblock {\em Algebraic groups and their birational invariants}, volume 179 of
		{\em Translations of Mathematical Monographs}.
		\newblock American Mathematical Society, Providence, RI, 1998.
		\newblock Translated from the Russian manuscript by Boris Kunyavski [Boris \`E.
		Kunyavski\u{\i}].
		
	\end{thebibliography}
\end{document}